\documentclass{amsart}

\usepackage[normalem]{ulem}

\usepackage{lineno}

\usepackage{caption}

\usepackage{soul}

\usepackage{graphicx}

\usepackage{amssymb, amsmath, amsthm, hyperref, euscript, color}
\usepackage[dvipsnames]{xcolor}

\usepackage{amscd}

\usepackage{bbm}

\usepackage{enumitem}

\usepackage[english]{babel}

\numberwithin{equation}{section}

\theoremstyle{plain}

\newtheorem{theorem}{Theorem}

\newtheorem{proposition}[theorem]{Proposition}

\newtheorem{remark}[theorem]{Remark}

\newtheorem{lemma}[theorem]{Lemma}

\newtheorem{thm}{Theorem}

\newtheorem{qst}[thm]{Question}

\newtheoremstyle{named}{}{}{\itshape}{}{\bfseries}{.}{.5em}{\thmnote{#3 }#1}
\theoremstyle{named}
\newtheorem*{namedconjecture}{Conjecture}

\theoremstyle{definition}

\newtheorem{definition}[theorem]{Definition}

\newcommand{\R}{\mathbb{R}}

\newcommand{\CP}{{\mathbb C\mkern-0.5mu\mathrm P}}

\newcommand{\RP}{{\mathbb R\mkern-0.5mu\mathrm P}}

\newcommand{\Z}{\mathbb{Z}}

\newcommand{\N}{\mathbb{N}}

\DeclareMathOperator{\Gl}{GL}

\DeclareMathOperator{\Pin}{Pin}

\DeclareMathOperator{\Sl}{SL}
\DeclareMathOperator{\cs}{\#}

\makeatletter
\@namedef{subjclassname@2020}{%
  \textup{2020} Mathematics Subject Classification}
\makeatother

\begin{document}

\author{Rafael Torres}

\dedicatory{Feliz cumplea\~nos, Bob.}

\title[New collections of homotopy $S^4s$ and $\RP^4s$]{A novel construction of homotopy 4-spheres and real projective 4-spaces \`a la Cappell-Shaneson.}

\address{Scuola Internazionale Superiore di Studi Avanzati (SISSA)\\ Via Bonomea 265\\34136\\Trieste\\Italy}

\email{rtorres@sissa.it}

\subjclass[2020]{Primary 57R55; Secondary 57M60}

\maketitle

\emph{Abstract}: We produce new collections of homotopy 4-spheres and homotopy real projective 4-spaces by extending a classical construction due to Cappell-Shaneson. We use the extension to devise a mechanism that produces sets of smooth 4-manifolds within a given homeomorphism class that satisfies a mild topological condition. 

\section{Introduction.}\label{Introduction}

Cappell-Shaneson homotopy 4-spheres \cite{[CappellShaneson]} have a decades long tradition of being a propitious source of counter-examples to the \begin{namedconjecture}[Smooth 4D Poincar\'e]Every homotopy 4-sphere is diffeomorphic to $S^4$.\end{namedconjecture} These homotopy 4-spheres arise by surgerying the 0-section of a 3-torus bundle over the circle whose monodromy is given by an element $A\in \Sl(3, \Z)$ that satisfies $\det(A - I) = 1$. In the sequel, we call any such matrix a CS matrix; see Section \ref{Section Mapping Tori and CS Spheres}. A considerable amount of effort has been invested by several mathematicians to understand these 4-manifolds better. Thanks to foundational work of Aitchison-Rubinstein \cite{[AitchisonRubinstein]}, Akbulut \cite{[Akbulut1]}, Akbulut-Kirby \cite{[AkbulutKirby1], [AkbulutKirby2]} and Gompf \cite{[Gompf1], [Gompf2], [Gompf3]}, we are now aware that a large portion of the set of Capell-Shaneson homotopy 4-spheres consists of elements that are actually diffeomorphic to $S^4$; see \cite[Section 1]{[Gompf3]} for a historical account. 

The first main contribution of this note is to amplify considerably the size of the set of plausible  counter-examples to the Smooth 4D Poincar\'e conjecture through a generalization of Cappell-Shaneson's construction. 

\begin{thm}\label{Theorem CS Examples}For any collection $\{A_1, \ldots, A_k\}$ of CS matrices with $k\geq 1$, there is a pair of homotopy 4-spheres $\Sigma^{\epsilon}_{A_1\cdots A_k}$ determined by a choice of framing $\epsilon \in \Z/2$ of a surgery along the 0-section of a $(\underbrace{T^3\cs\cdots\cs T^3}_k)$-bundle over $S^1$. 
\end{thm}

The case $k = 1$ corresponds to the Cappell-Shaneson homotopy 4-spheres, which we call CS homotopy 4-spheres from now on; see Section \ref{Section Generalizations} for a brief description of other constructions of homotopy 4-spheres. Are the smooth structures on the 4-spheres of Theorem \ref{Theorem CS Examples} standard?

The CS homotopy 4-spheres arise from Cappell-Shaneson's construction of an inequivalent smooth structure on the real projective 4-space in \cite{[CappellShaneson]}. There are two smooth s-cobordism classes of 4-manifolds homeomorphic to $\RP^4$ as shown by Ruberman \cite{[Ruberman]} and Cappell-Shaneson's construction produces at most two inequivalent smooth structures as shown by Aitchison-Rubinstein \cite{[AitchisonRubinstein]}. To the best of our knowledge, there is no example in the literature of a closed nonorientable 4-manifold that is known to admit at least three inequivalent smooth structures, let alone infinitely many. The main contribution of this paper is to produce explicit constructions of 4-manifolds that might not be diffeomorphic. Our extension of Cappell-Shaneson's work yields a myriad of candidates for such examples as our second main result illustrates on the homeomorphism class of $\RP^4$. The connected sum of $k$ copies of $S^2\times S^2$ is denoted by $k\cdot(S^2\times S^2)$.

\begin{thm}\label{Theorem Infinite RP4s}There is a set $\{Q_r: r\in \N\}$ of closed 4-manifolds that satisfies the following properties.\begin{itemize}\item There is a homeomorphism $Q_r\rightarrow \RP^4$ for every $r\in \N$.
\item There is no diffeomorphism\begin{equation*}Q_r\cs (n - 1)\cdot(S^2\times S^2)\rightarrow \RP^4\cs (n - 1)\cdot(S^2\times S^2)\end{equation*}for any $n\in \N$ provided $r \neq 0 \mod 2$.
\item There is an $N\in \N$ such that there is a diffeomorphism\begin{equation*}Q_r\cs N\cdot(S^2\times S^2)\rightarrow \RP^4\cs N\cdot(S^2\times S^2)\end{equation*}if $r = 0 \mod 2$.
\end{itemize}
\end{thm}

It is intriguing to ask whether the 4-manifolds of Theorem \ref{Theorem Infinite RP4s} are pairwise non-diffeomorphic or if the set contains more than two elements that are pairwise non-diffeomorphic. We hope the explicit nature of our construction stimulates further work as we gather several related questions in Section \ref{Section Questions}. Theorem \ref{Theorem Infinite RP4s} does recover results of Cappell-Shaneson and Fintushel-Stern \cite{[FintushelStern]}, and their universal covers are included in the homotopy 4-spheres of Theorem \ref{Theorem CS Examples}. We use the 4-manifolds constructed in the proof of Theorem \ref{Theorem Infinite RP4s} to devise a broader cut-and-paste mechanism to unveil inequivalent smooth structures on nonorientable 4-manifolds that satisfy certain mild restrictions. The twisted nonorientable 2-disk bundle over the real projective plane is denoted by $D^2\widetilde{\times} \RP^2$; see Section \ref{Section CS RP4s} for a definition. 

\begin{thm}\label{Theorem Mechanism}Let $X$ be a closed nonorientable 4-manifold that contains a smoothly embedded real projective plane $P\subset X$ with tubular neighborhood $\nu(P) = D^2\widetilde{\times} \RP^2$. There is a set $\{X_r: r\in \N\}$ and a homeomorphism\begin{equation}\label{Homeomorphism Infinite Set}X_r\rightarrow X\end{equation}for every $r\in \N$.

Suppose furthermore that $H^1(X; \Z/2) = \Z/2$ and that its second Stiefel-Whitney class satisfies $w_2(X) = 0$. If $r\neq 0 \mod 2$ and the $\eta$-invariant of $X$ satisfies $\eta(X, \phi_X)\neq \pm \frac{1}{2}\mod 2\Z$ for a $\Pin^+$-structure $(X, \phi_X)$, then  there is no diffeomorphism\begin{equation}X_r\cs (n - 1)\cdot(S^2\times S^2)\rightarrow X\cs (n - 1)\cdot(S^2\times S^2)\end{equation} for any $n\in \N$.\end{thm}

Cappell-Shaneson's work \cite[Theorem 3.1]{[CappellShaneson]} yields a simple homotopy equivalence with nontrivial normal invariant $M\rightarrow N$ between a 4-manifold $M$ that contains an orientation-reversing simple loop with homotopy class of order two in $\pi_1(M)$ and the 4-manifold $N$ produced in their original construction; see Section \ref{Section CS RP4s}. A contribution of the mechanism summarized in Theorem \ref{Theorem Mechanism} is that it provides a homeomorphism between the 4-manifolds produced through an indirect use of the topological classification results of homotopy $\RP^4$s due to Ruberman \cite{[Ruberman]} and Hambleton-Kreck-Teichner \cite{[HambletonKreckTeichner]}. The vanishing of the second Stiefel-Whitney class $w_2(X)$ enables the usage of the $\eta$-invariant and this spectral invariant is employed to distinguish the smooth structures through work of Stolz \cite{[Stolz]}; see Section \ref{Section New HRP4s} and Remark \ref{Remark Hypothesis H1}.


The structure of the paper is the following. Section \ref{Section Mapping Tori and CS Spheres} contains a brief description of the raw materials and technology used in the construction of the CS homotopy 4-spheres. Cappell-Shaneson's construction of mapping tori using the 3-torus is extended in Section \ref{Section Extension CS Mapping Tori} to a prescribed number of connected sums of $T^3$. The latter 4-manifolds are used in Section \ref{Section Proof of Theorem CS Examples} to build new homotopy 4-spheres and to prove Theorem \ref{Theorem CS Examples}. Section \ref{Section Generalizations} contains a discussion on other constructions of homotopy 4-spheres that can be paired with our extension in order to produce even larger collections. Our construction of new homotopy $\RP^4$s is described in Section \ref{Section New HRP4s}. The proof of Theorem \ref{Theorem Infinite RP4s} is in Section \ref{Section Proof of Theorem Infinite RP4s}. Theorem \ref{Theorem Mechanism} is proven in Section \ref{Section Construction Mechanism}. We record several canonical questions in Section \ref{Section Questions} for the convenience of the interested reader.\\


\section{Mapping tori and Cappell-Shaneson homotopy 4-spheres.}\label{Section Mapping Tori and CS Spheres} Let $\varphi_A: T^3\rightarrow T^3 = \R^3/\Z^3$ be a self-diffeomorphism of the 3-torus such that the induced map in homology $(\varphi_A)_\ast: H_1(T^3)\rightarrow H_1(T^3)$ is given by a CS matrix $A\in \Sl(3; \Z)$. Cappell-Shaneson \cite{[CappellShaneson]} construct the mapping torus associated to the CS matrix $A$\begin{equation}\label{Mapping Torus}M_A := T^3\times [0, 1]/\sim\end{equation}under the identification $(x, 0) \sim (\varphi_A(x), 1)$. The 4-manifold (\ref{Mapping Torus}) is the total space of a $T^3$-bundle over $S^1$ with monodromy given by the CS matrix $A$. A section of $M_A$ is given by the simple loop\begin{equation}\label{Section}\alpha_A:= (\{x_0\}\times [0, 1])/\sim)\subset M_A.\end{equation}The sections of $M_A$ are related by a self-diffeomorphism $M_A\rightarrow M_A$ that preserves its fibers.

\begin{lemma}\label{Lemma CS}\cite{[CappellShaneson]}. The aspherical 4-manifold (\ref{Mapping Torus}) is a homology $S^1\times S^3$ and its fundamental group has a presentation\begin{equation}\pi_1(M_A) = \langle x_1, x_2, x_3, t : tx_i t^{-1} = A x_i, [x_i, x_i] = 1\rangle\end{equation}for $i, j = 1, 2, 3$. The elements $\{x_1, x_2, x_3\}$ are homotopy classes of closed simple loops $\{\alpha_1, \alpha_2, \alpha_3\}$ in the 3-torus fiber of $M_A$ and $t$ is the homotopy class of a section (\ref{Section}).\end{lemma}

There are two choices of framing for the loop $\alpha_A\subset M_A$ parametrized by $\epsilon \in \Z/2$. The untwisted framing corresponds to $\epsilon = 0$ and the twisted one to $\epsilon = 1$; see Aitchison-Rubinstein \cite{[AitchisonRubinstein]} and Gompf \cite[Section 4]{[Gompf3]}.

\begin{definition}\label{Definition CS 4-Sphere}A Cappell-Shaneson homotopy 4-sphere associated to a CS matrix $A$ is the pair of 4-manifolds\begin{equation}\Sigma_A^\epsilon:= (M_A\setminus \nu(\alpha_A)) \cup_\epsilon (D^2\times S^2).\end{equation}\end{definition}Both CS homotopy 4-spheres of Definition \ref{Definition CS 4-Sphere} are homeomorphic to $S^4$ for every choice of CS matrix \cite{[Freedman]}, \cite[Theorem 1.2.27]{[GompfStipsicz]}. 

\section{An assemblage of orientable $(T^3\cs \cdots \cs T^3)$-bundles over $S^1$.}\label{Section Extension CS Mapping Tori}Let\begin{equation}\label{CS Matrices Collection}\{A_1, \ldots, A_k\}\end{equation}be a collection of CS matrices as defined in the Introduction and let\begin{equation}\label{CS Collection}\{M_{A_1}, \ldots, M_{A_k}\}\end{equation}be the corresponding mapping tori associated to the CS matrices (\ref{CS Matrices Collection}) that were built in Section \ref{Section Mapping Tori and CS Spheres}. Let $M_{A_1\cdots A_k}$ be the closed oriented 4-manifold obtained from inductively gluing together each element of (\ref{CS Collection}) along the a section of $M_{A_i}$ for $1\leq i \leq k$ as follows. Let $\alpha_1\subset M_{A_1}$ be a section as in (\ref{Section}) and $\alpha_i, \alpha_i'\subset M_{A_i}$ be disjoint sections for $2\leq i \leq k$. Assemble the 4-manifold\begin{equation}\label{Generalized MT}M_{A_1\cdots A_k}:= (M_{A_1}\setminus \nu(\alpha_1))\cup (M_{A_2}\setminus \nu(\alpha_2)\sqcup \nu(\alpha_2'))\cup \cdots \cup M_{A_k}\setminus \nu(\alpha_k)\end{equation}that contains a closed simple loop\begin{equation}\label{Section Generalized}\alpha_k'\subset M_{A_1\cdots A_k}.\end{equation}

An extension of Lemma \ref{Lemma CS} to this setting is as follows. The connected sum of $k$ copies of the 3-torus is denoted by $k\cdot T^3$

\begin{lemma}\label{Lemma Extension CS}The 4-manifold $M_{A_1\cdots A_k}$ defined in (\ref{Generalized MT}) is a $(k\cdot T^3)$-bundle over $S^1$ with sections isotopic to the loop (\ref{Section Generalized}) and it has the homology of $S^1\times S^3$. Its fundamental group has a presentation with generators\begin{equation}\label{Homotopy classes fibers}\{t, \{x^1_1, x^1_2, x^1_3\}, \cdots, \{x^k_1, x^k_2, x^k_3\}\}\end{equation} subject to the relations\begin{equation}\label{Generalized relations}\{tx^l_it^{-1} = Ax^l_i: 1\leq i \leq 3, 1\leq l \leq k\}\end{equation}as well as $[x^l_i, x^l_j] = 1$. The elements\begin{equation}\{\{x^1_1, x^1_2, x^1_3\}, \cdots, \{x^k_1, x^k_2, x^k_3\}\}\end{equation}are homotopy classes of closed simple loops\begin{equation}\{\{\alpha^1_1, \alpha^1_2, \alpha^1_3\}, \ldots, \{\alpha^k_1, \alpha^k_2, \alpha^k_3\}\}\end{equation}in the $(k\cdot T^3)$-fiber of $M_k$ and $t$ is the homotopy class of the closed simple loop $\alpha_k'\subset M_{A_1\cdots A_k}$.\end{lemma}

\begin{proof}The bundle structure of the 4-manifold $M_{A_1\cdots A_k}$ follows from its assembly of $T^3$-bundles done in (\ref{Generalized MT}). A general position argument shows that the inclusions $M\setminus \nu(\alpha)\hookrightarrow M$ and $M\setminus \nu(\alpha)\sqcup \nu(\alpha')\hookrightarrow M$ induce isomorphisms of the corresponding fundamental groups since loops in a 4-manifold $M$ have codimension three. Lemma \ref{Lemma CS} and the Seifert-van Kampen imply the existence of a presentation for $\pi_1(M_{A_1\cdots A_k})$ with generators (\ref{Homotopy classes fibers}) and relations (\ref{Generalized relations}) along with the corresponding commutators. Since the homotopy class of a loop in a 4-manifold determines its isotopy class \cite[Example 4.1.3]{[GompfStipsicz]}, the loop $\alpha_k'\subset M_{A_1\cdots A_k}$ is isotopic to a section.\end{proof}

\section{New homotopy 4-spheres and a proof of Theorem \ref{Theorem CS Examples}.}\label{Section Proof of Theorem CS Examples}

We prove the first result discussed in the Introduction in this section. For every $(k\cdot T^3)$-bundle over the circle $M_{A_1\cdots A_k}$ that decomposes into $k\geq 1$ copies of $T^3$-bundles over $S^1$ whose monodromy is a CS matrix as in the previous section, there is a pair of 4-manifolds\begin{equation}\label{New Homotopy 4-Spheres}\Sigma^\epsilon_{A_1\cdots A_k} = (M_{A_1\cdots A_k}\setminus \nu(\alpha_k')) \cup_\epsilon (D^2\times S^2)\end{equation}that are determined by a choice of framing $\epsilon \in \Z/2\Z$. Theorem \ref{Theorem CS Examples} is equivalent to the following result.

\begin{theorem}\label{Theorem Extension CS}Let $\Sigma^\epsilon_{A_1\cdots A_k}$ be the closed 4-manifold obtained by doing surgery to $M_{A_1\cdot A_k}$ along the simple loop $\alpha_k'\subset M_{A_1\cdots A_k}$ with a choice of framing $\epsilon \in \Z/2$ as in (\ref{New Homotopy 4-Spheres}). There is a homeomorphism\begin{equation*}\Sigma^\epsilon_{A_1\cdots A_k}\rightarrow S^4.\end{equation*}\end{theorem}

\begin{proof}The fundamental group of (\ref{New Homotopy 4-Spheres}) is trivial regardless of the choice of framing by Lemma \ref{Lemma Extension CS} and the Seirfert-van Kampen theorem since $\{A_1, \ldots, A_k\}$ are CS matrices. It is straight-forward to see that the Euler characteristic is $\chi(\Sigma^\epsilon_{A_1\cdots A_k}) = \chi(M_{A_1\cdots A_k}\setminus \nu(\alpha_k')) + \chi(D^2\times S^2) = 2$ for both choices $\epsilon \in \Z/2\Z$. Freedman's classification result \cite{[Freedman]} now implies Theorem \ref{Theorem Extension CS}.\end{proof}

\subsection{A brief discussion on other constructions of homotopy 4-spheres}\label{Section Generalizations}There are other sources of counter-examples to the Smooth 4D Poincar\'e Conjecture of a different nature to the ones produced in this paper. We finish this section with a brief description of them. In \cite[Section 3]{[Nash]}, Nash constructed a collection of homotopy 4-spheres by performing torus surgeries to a union $(T'\times T')\cup (T'\times T')$ where $T'$ is the punctured 2-torus $T' = T^2\setminus D^2$. He also studied a relation between his examples and the CS homotopy 4-spheres. Akbulut showed that a large subset of Nash's examples are diffeomorphic to the 4-sphere in \cite{[Akbulut2]}.

Calegari \cite{[Calegari]} discribed a construction of further examples of homotopy 4-spheres from fibered knots in $S^3$; see \cite[Section 2]{[ChaKim]}. Analogously to the CS homotopy 4-spheres, Calegari's examples are obtained by performing surgery along a section of $k\cdot (S^1\times S^2)$-bundles over $S^1$, where $k\cdot (S^1\times S^2)$ is the connected sum of $k$ copies of $S^1\times S^2$. Calegari's examples have been shown to be diffeomorphic to $S^4$ by Cha-Kim \cite{[ChaKim]}. 

There are other ways to produce even more homotopy 4-spheres and we now describe a couple of them. The examples collected in Theorem \ref{Theorem CS Examples} suggest the existence of further homotopy 4-spheres that arise from an extension of Calegari's construction to fibered links in the 3-sphere. Further possible counter-examples to the Smooth 4D Poincar\'e conjecture are also produced through surgery on a 4-manifold obtained from gluing the bundles produced in this paper along with those obtained from Nash's and/or Calegari's constructions.

\section{New homotopy real projective 4-spaces.}\label{Section New HRP4s}

The tools used in the proof of Theorem \ref{Theorem Infinite RP4s} are described in this section. We spend the first two sections recalling the original construction of homotopy $\RP^4$s due to Cappell-Shaneson.

\subsection{Cappell-Shaneson's construction of two nonorientable mapping tori}\label{Section Nonorientable CS Construction}Let $\varphi_B: T^3\rightarrow T^3 = \R^3/\Z^3$ be a self-diffeomorphism such that the induced map $(\varphi_B)_\ast: H_1(T^3)\rightarrow H_1(T^3)$ is given by a matrix $B\in \Gl(3; \Z)$ with determinant $\det(B) = - 1$ and that satisfies $\det(B^2 - I) = 1$. Any such matrix is caled a NCS matrix in the sequel. Aitchison-Rubinstein \cite[Lemma 5.2]{[AitchisonRubinstein]} have shown that thare are exactly two conjugacy classes of such matrices with representatives given by\begin{center}$B^\ast =\left( \begin{array}{ccc}
0 & 1 & 0   \\
0 & 0 & 1  \\
-1 & 1 & 0
\end{array} \right)$ and $B^{\ast \ast} =\left( \begin{array}{ccc}
0 & 1 & 0   \\
0 & 0 & 1  \\
-1 & 1 & 2
\end{array} \right)$.\end{center}

Construct the mapping torus\begin{equation}\label{Nonorientable Mapping Torus}M_B := T^3\times [0, 1]/\sim\end{equation}under the identification $(x, 0) \sim (\varphi_B(x), 1)$. The 4-manifold (\ref{Nonorientable Mapping Torus}) is the non-orientable total space of a $T^3$-bundle over $S^1$ with monodromy $B\in \Gl(3; \Z)$. A section of $M_B$ is denoted by\begin{equation}\label{Nonorientable Section}\alpha_B:= (\{x_0\}\times [0, 1])/\sim)\subset M_B.\end{equation}There are exactly two such mapping tori distinguished by the conjugacy classes of their monodromy. The homology and fundamental group of these 4-manifolds are collected in the next lemma, where the nontrivial 3-sphere bundle over the circle is denoted by $S^3\widetilde{\times} S^1$. 

\begin{lemma}\label{Lemma CSAT}Cappell-Shaneson \cite[Proposition 2.1 (a)]{[CappellShaneson]}. The aspherical 4-manifold put together in (\ref{Nonorientable Mapping Torus}) is a homology $S^3\widetilde{\times} S^1$ and its fundamental group has a presentation\begin{equation}\pi_1(M_B) = \langle x_1, x_2, x_3, t : tx_i t^{-1} = B x_i, [x_i, x_i] = 1\rangle\end{equation}for $i, j = 1, 2, 3$. The elements $\{x_1, x_2, x_3\}$ are homotopy classes of closed simple loops $\{\alpha_1, \alpha_2, \alpha_3\}$ in the 3-torus fiber of $M_B$ and $t$ is the homotopy class of a section (\ref{Nonorientable Section}).



\end{lemma}

There is a $\Pin^+$-structure $(M_B, \phi_{M_B})$ on the tangent bundle of $M_B$ \cite[Corollary 6.4]{[Stolz]}, which enables the use of the $\eta$-invariant; see \cite[Section 4]{[Stolz]} for a description. Stolz used this spectral invariant to pin down the class of (\ref{Nonorientable Mapping Torus}) in the fourth $\Pin^+$-bordism group, which was computed to be isomorphic to $\Z/16\Z$ by Kirby-Taylor \cite[Theorem 5.2]{[KirbyTaylor]}.

\begin{proposition}\label{Proposition MT Stolz}Stolz \cite[Proposition 7.3]{[Stolz]}. There is a $\Pin^+$-structure $(M_B, \phi_{M_B})$ such that $\eta(M_B, \phi_{M_B}) = 1 \mod 2\Z$ and\begin{equation}\label{Pin Cobordism Class MB}[(M_B, \phi_{M_B})] = 8 \in \Omega^{\Pin^+}_4\cong \Z/16\Z.\end{equation} 
\end{proposition}

Since $H^1(M_B; \Z/2\Z) = \Z/2\Z$, there are exactly two $\Pin^+$-structures on $M_B$ \cite[Corollary 6.4]{[Stolz]}. Proposition \ref{Proposition MT Stolz} holds for both structures.

\subsection{Cappell-Shaneson homotopy $\RP^4$s}\label{Section CS RP4s}Using the mapping tori described in the previous section, Cappell-Shaneson \cite{[CappellShaneson]} constructed a closed 4-manifold $Q_1$ that is not diffeomorphic to the real projective 4-space $\RP^4$ even after taking a connected sum with an arbitrary number of copies of $S^2\times S^2$ as follows. Let $\alpha\subset \RP^4$ be a closed simple loop whose homotopy class generates the group $\pi_1(\RP^4) = \Z/2\Z$. Cappell-Shaneson assemble a 4-manifold\begin{equation}\label{CS RP4}Q:= (\RP^4\setminus \nu(\alpha))\cup (M_B\setminus \nu(\alpha_B)),\end{equation}where $\alpha_B\subset M_B$ is the loop (\ref{Nonorientable Section}). The diffeomorphism used to identify the compact pieces along their boundaries in (\ref{CS RP4}) is chosen so that it matches the induced $\Pin^+$-structures on the common $S^2\widetilde{\times} S^1$ boundaries \cite[p. 160]{[Stolz]}, \cite{[AitchisonRubinstein], [CappellShaneson]}.

\begin{remark}One of the two s-cobordism classes of 4-manifolds homeomorphic to $\RP^4$ is realized by the standard smooth structure, while the other one is realized by (\ref{CS RP4}) with either choice of matrix $B$.\end{remark}

Lemma \ref{Lemma CS} and the Seifert-van Kampen theorem imply that $\pi_1(Q) = \Z/2\Z$. Cappell-Shaneson show that there is a simple homotopy equivalence $Q \rightarrow \RP^4$ in \cite[Theorem 3.1]{[CappellShaneson]}. Using Freedman's results, Ruberman \cite{[Ruberman]} showed that the 4-manifold (\ref{CS RP4}) is homeomorphic to $\RP^4$. Akbulut \cite{[Akbulut1]} and Gompf \cite{[Gompf1], [Gompf3]} proved that the universal cover of $Q$ is diffeomorphic to the 4-sphere. We summarize the foundational work of these mathematicians in the next theorem. 

\begin{theorem}\label{Theorem Raw Material}Cappell-Shaneson, Ruberman, Akbulut, Gompf. There is a closed nonorientable 4-manifold $Q$ that satisfies the following properties. 
\begin{itemize}
\item There is a homeomorphism $Q\rightarrow \RP^4$.
\item There is no diffeomorphism $Q\cs (n - 1)\cdot(S^2\times S^2)\rightarrow \RP^4\cs (n - 1)\cdot(S^2\times S^2)$ for any $n\in \N$.
\item The universal cover of $Q$ is diffeomorphic to $S^4$. 
\end{itemize}
\end{theorem}

Instead of Cappell-Shaneson's invariant \cite{[CappellShaneson]}, we will use the $\eta$-invariant to distinguish smooth structures in this paper; the reader is directed towards \cite[Sections 3, 4, and 5]{[Stolz]} for a description of this spectral invariant. Recall that two 4-manifolds $X$ and $Y$ are stably diffeomorphic if there is an integer $N\in \Z_{\geq 0}$ such that $X\cs N\cdot(S^2\times S^2)$ is diffeomorphic to $Y\cs N\cdot(S^2\times S^2)$. Stolz \cite[Theorem B]{[Stolz]} showed that the Euler characteristic and the $\eta$-invariant determine the stable diffeomorphism class of closed nonorientable 4-manifolds with fundamental group of order two that admit a $\Pin^+$-structure. We collect the corresponding values for the smooth structures discussed so far in the next theorem.

\begin{theorem}\label{Theorem Stolz}Stolz \cite{[Stolz]}. The values of the $\eta$-invariant for $\RP^4$ and $Q$ are\begin{equation*}\eta(\RP^4, \pm \phi_{\RP^4}) = \pm \frac{1}{8} \mod 2\Z
\end{equation*}and\begin{equation*}\eta(Q, \pm \phi_{Q}) = \pm \frac{7}{8} \mod 2\Z.\end{equation*}

In particular, there is no diffeomorphism\begin{equation*}Q\cs (n - 1)\cdot(S^2\times S^2)\rightarrow \RP^4\cs (n - 1)\cdot(S^2\times S^2)\end{equation*} for any $n\in \N$.
\end{theorem}

A couple of words are in order regarding the notation used in Theorem \ref{Theorem Stolz} for the $\Pin^+$-structures on the homotopy $\RP^4$s. There is a bijection between $H^1(X; \Z/2)$ and the number of $\Pin^+$-structures on a closed nonorientable 4-manifold $X$ with $w_2(X) = 0$ \cite[Lemma 1]{[HambletonKreckTeichner]}, \cite[Corollary 6.4]{[Stolz]}. If $H^1(X; \Z/2) = \Z/2$, there are two such structures and they are mutual inverses in the fourth $\Pin^+$-bordism group $\Omega^{\Pin^+}_4$ \cite[p. 190]{[KirbyTaylor]}.

The inequivalent smooth structure on $\RP^4$ can be phrased as an inequivalent smooth structure on the compact 4-manifold\begin{equation}\label{Tubular Neighborhood RP2}D^2\widetilde{\times} \RP^2 = \frac{D^2\times S^2}{(d, x)\sim (-d, \mathbb{A}(x))}\end{equation}for $(d, x)\in D^2\times S^2$ where $\mathbb{A}: S^2\rightarrow S^2$ is the antipodal involution. The 4-manifold (\ref{Tubular Neighborhood RP2}) is the twisted nonorientable 2-disk bundle over $\RP^2$ with boundary\begin{equation}\partial(D^2\widetilde{\times} \RP^2) = S^2\widetilde{\times} S^1.\end{equation}Let $\alpha_{Q}\subset Q$ be a closed simple orientation-reversing loop whose homotopy class generates $\pi_1(Q) = \Z/2\Z$. The compact 4-manifold $Q\setminus \nu(\alpha_Q)$ is homeomorphic, but not diffeomorphic to $D^2\widetilde{\times} \RP^2$, i.e., it is an inequivalent smooth structure on (\ref{Tubular Neighborhood RP2}).





\subsection{Further nonorientable mapping tori}\label{Section Extension Mapping Tori}We use the two 4-manifolds constructed in (\ref{Nonorientable Mapping Torus}) in this section to assemble a set of $(T^3\cs \cdots \cs T^3)$-bundles over $S^1$ with fiber a prescribed number of connected sums of 3-tori using the procedure of Section \ref{Section Extension CS Mapping Tori}. Let\begin{equation}\label{NCS Matrices Collection}\{B_1, \ldots, B_r\}\end{equation}be a collection of NCS matrices (see Section \ref{Section Nonorientable CS Construction}) and let\begin{equation}\label{NCS Collection}\{M_{B_1}, \ldots, M_{B_r}\}\end{equation}be the corresponding mapping tori associated to the NCS matrices (\ref{NCS Matrices Collection}). Let $M_{B_1\cdots B_r}$ be the closed oriented 4-manifold obtained from inductively gluing together each element of (\ref{NCS Collection}) along a section of $M_{B_i}$ for $1\leq i \leq r$ as follows. Let $\alpha_1\subset M_{B_1}$ be a section as in (\ref{Nonorientable Section}) and $\alpha_i, \alpha_i'\subset M_{B_i}$ be disjoint sections for $2\leq i \leq r$. Assemble the 4-manifold\begin{equation}\label{Generalized NMT}M_{B_1\cdots B_r}:= (M_{B_1}\setminus \nu(\alpha_1))\cup (M_{B_2}\setminus \nu(\alpha_2)\sqcup \nu(\alpha_2'))\cup \cdots \cup (M_{B_r}\setminus \nu(\alpha_r))\end{equation}that contains a closed simple orientation-reversing loop\begin{equation}\label{Section NGeneralized}\alpha_r'\subset M_{B_1\cdots B_r}.\end{equation}Every gluing diffeomorphisms in (\ref{Generalized NMT}) is chosen so that it preserves the bundle structure of the common $S^2\widetilde{\times} S^1$ boundaries. 

The canonical extension of Lemma \ref{Lemma Extension CS} to this setting is as follows.

\begin{lemma}\label{Lemma Generalized MT}The 4-manifold $M_{B_1\cdots B_k}$ defined in (\ref{Generalized NMT}) is a $(r\cdot T^3)$-bundle over $S^1$ with sections isotopic to the loop (\ref{Section NGeneralized}) and it has the homology of $S^3\widetilde{\times} S^1$. Its fundamental group has a presentation with generators\begin{equation}\label{Homotopy classes Nfibers}\{t, \{x^1_1, x^1_2, x^1_3\}, \cdots, \{x^r_1, x^r_2, x^r_3\}\}\end{equation} subject to the relations\begin{equation}\label{Generalized Nrelations}\{tx^l_it^{-1} = Bx^l_i: 1\leq i \leq 3, 1\leq l \leq r\}\end{equation}as well as $[x^l_i, x^l_j] = 1$. The elements\begin{equation}\{\{x^1_1, x^1_2, x^1_3\}, \cdots, \{x^r_1, x^r_2, x^r_3\}\}\end{equation}are homotopy classes of closed simple loops\begin{equation}\{\{\alpha^1_1, \alpha^1_2, \alpha^1_3\}, \ldots, \{\alpha^r_1, \alpha^r_2, \alpha^r_3\}\}\end{equation}in the $(r\cdot T^3)$-fiber of $M_{B_1\cdots B_k}$ and $t$ is the homotopy class of the closed simple loop $\alpha_k'\subset M_{B_1\cdots B_r}$
\end{lemma}

We finish this section by recording the values of the $\eta$-invariant of the mapping tori (\ref{Generalized NMT}) and their dependence on the value of $r\in \N$.

\begin{proposition}\label{Proposition Pin Structures Mapping Torus Generalized}There are $\Pin^+$-structures $(M_{B_1\cdots B_r}, \pm \phi_{M_{B_1\cdots B_r}})$ such that\begin{equation}\label{Two Pin Structures Mapping Torus}
  \eta(M_{B_1\cdots B_r}, \pm \phi_{M_r}) = \begin{cases}
    0 \mod 2\Z, & \text{if $r$ is even.}\\
   1 \mod 2\Z, & \text{if $r$ is odd}.
  \end{cases}
\end{equation}\end{proposition}

\begin{proof}By \cite[Corollary 6.4]{[Stolz]}, there are two $\Pin^+$-structures $(M_{B_1\cdots B_r}, \pm \phi_{M_{B_1\cdots B_r}})$ given that $H^ 1(M_{B_1\cdots B_r}; \Z/2\Z) = H^1(S^3\widetilde{\times}S^1; \Z/2\Z) = \Z/2\Z$. The values (\ref{Two Pin Structures Mapping Torus}) hold by the existence of a $\Pin^+$-bordism $(V; M_{B_1\cdots B_r}, M_{B_1}\sqcup \cdots \sqcup M_{B_r})$, the values of Proposition \ref{Proposition MT Stolz} and the property of the $\eta$-invariant of being additive with respect to disjoint unions \cite[Section 7]{[Stolz]}.\end{proof}

\subsection{New homotopy $\RP^4$s}The Cappell-Shaneson construction can be altered to produce (for any $r > 0$) a 4-manifold $Q_r$ that is homeomorphic to $\RP^4$ using the mapping tori described in Section \ref{Section Extension Mapping Tori} as follows. Consider again the closed orientation-reversing simple loop  $\alpha\subset \RP^4$ whose homotopy class generates the fundamental group of the real projective 4-space. For a given $r\in \N$, assemble a 4-manifold\begin{equation}\label{Generalized CS RP4}Q_r:= (\RP^4\setminus \nu(\alpha))\cup (M_{B_1\cdots B_r})\setminus \nu(\alpha_r')),\end{equation}where $\alpha_r'\subset M_{B_1\cdots B_r}$ is the loop (\ref{Section NGeneralized}). The gluing diffeomorphism in (\ref{Generalized CS RP4}) is chosen so that it identifies the induced $\Pin^+$-structures on the common $S^2\widetilde{\times} S^1$ boundaries \cite[p. 160]{[Stolz]}( cf. \cite{[AitchisonRubinstein], [CappellShaneson]}).

\begin{lemma}\label{Lemma Homeomorphism}For every $r\in \N$, there is a homeomorphism $Q_r\rightarrow \RP^4$.
\end{lemma}

\begin{proof} The Seifert-van Kampen theorem and Lemma \ref{Lemma Generalized MT} imply that the fundamental group is $\pi_1(Q_r) = \Z/2\Z$ for every $r\in \N$. A straight-forward computation in homology reveals that $\chi(Q_r) = 1$. The classification results of Hambleton-Kreck-Teichner \cite{[HambletonKreckTeichner]} imply that $Q_r$ is homeomorphic to the real projective 4-space for every value of $r$.
\end{proof}

We now gather the values of the $\eta$-invariant of (\ref{Generalized CS RP4}).

\begin{proposition}\label{Proposition Values Eta}There are $\Pin^+$-structures $(Q_r, \pm \phi_{Q_r})$ such that\begin{equation}\label{Two Pin Structures}
  \eta(Q_r, \pm \phi_{Q_r}) = \begin{cases}
    \pm \frac{1}{8} \mod 2\Z, & \text{if $r$ is even.}\\
    \pm \frac{7}{8} \mod 2\Z, & \text{if $r$ is odd}.
  \end{cases}
\end{equation}
The corresponding $\Pin^+$-bordism classes are\begin{equation}\label{Pin Cobordism Class Qr Even}[(Q_r, \pm \phi_{Q_r})] = \pm 1 \in \Omega^{\Pin^+}_4\cong \Z/16\Z.\end{equation}for even values of $r\in \N$ and\begin{equation}\label{Pin Cobordism Class Qr Odd}[(Q_r, \pm \phi_{Q_r})] = \pm 9 \in \Omega^{\Pin^+}_4\cong \Z/16\Z\end{equation}for odd values of $r\in \N$.\end{proposition}

\begin{proof}The result follows from Theorem \ref{Theorem Stolz} and Proposition \ref{Proposition Pin Structures Mapping Torus Generalized} given that there is a $\Pin^+$-bordism between $Q_r$ and $\RP^4\sqcup M_{B_1\cdots B_r}$ for every $r\in \N$ and the $\eta$-invariant is additive with respect to disjoint unions \cite[Section 7]{[Stolz]}.\end{proof}

\subsection{Proof of Theorem \ref{Theorem Infinite RP4s}}\label{Section Proof of Theorem Infinite RP4s}The existence of a homeomorphism between $Q_r$ and $\RP^4$ for any $r\in \N$ was established in Lemma \ref{Lemma Homeomorphism}. As it was mentioned before, Stolz \cite[Theorem B]{[Stolz]} showed that the Euler characteristic and the $\eta$-invariant determine the stable diffeomorphism class of a closed nonorientable 4-manifold with fundamental group of order two that admits a $\Pin^+$-structure. In particular, if $r = 0 \mod 2$, there is an integer $N\in \Z_{\geq 0}$ such that there is a diffeomorphism\begin{equation*}Q_r\cs N\cdot(S^2\times S^2)\rightarrow \RP^4\cs N\cdot(S^2\times S^2)\end{equation*} since $\eta(Q_r, \pm \phi_{Q_r}) = \eta(\RP^4, \pm \phi_{\RP^4}) \mod 2\Z$ by Proposition \ref{Proposition Values Eta}. On the other hand, if $r \neq 0 \mod 2$, these values differ by Proposition \ref{Proposition Values Eta} and there is no diffeomorphism\begin{equation*}Q_r\cs (n - 1)\cdot(S^2\times S^2)\rightarrow \RP^4\cs (n - 1)\cdot(S^2\times S^2)\end{equation*}for any $n\in \N$.\hfill $\square$

\section{Construction mechanism: Proof of Theorem \ref{Theorem Mechanism}.}\label{Section Construction Mechanism}

\begin{proof}Consider the compact 4-manifold $Q_r'= Q_r\setminus \nu(\alpha_r)$ obtained from the homotopy $\RP^4$ defined in (\ref{Generalized CS RP4}) by removing a tubular neighborhood of the loop $\alpha_r\subset Q_r$ whose homotopy class generates $\pi_1(Q_r) = \Z/2\Z$. Since $Q_r$ is homeomorphic to $\RP^4$, there is a homeomorphism\begin{equation}\label{Homeomorphism Block}Q_r'\rightarrow D^2\widetilde{\times} \RP^2\end{equation}as discussed at the end of Section \ref{Section CS RP4s}. For each $r\in \N$, assemble a 4-manifold\begin{equation}\label{Infinite Set R}X_r:= (X\setminus \nu(P))\cup Q_r'\end{equation}to produce a set $\{X_r: r\in \N\}$. There are two choices of gluing diffeomorphism (\ref{Infinite Set R}) by a result of Kim-Raymond \cite{[KimRaymond]}. The homeomorphism (\ref{Homeomorphism Block}) implies that the homeomorphism (\ref{Homeomorphism Infinite Set}) exists. Notice that the construction (\ref{Infinite Set R}) can be rephrased as\begin{equation}\label{Infinite Set R2}X_r = (X\setminus \nu(\alpha_X))\cup (M_{B_1\cdots B_r}\setminus \nu (\alpha_r')),\end{equation}for an orientation-reversing simple loop $\alpha_X\subset X$.

If the second Stiefel-Whitney class is $w_2(X) = 0$, then there is a $\Pin^+$-structure $(X, \phi_X)$. Since we also have a $\Pin^+$-structure $(M_{B_1\cdots B_r}, \phi_{M_{B_1\cdots B_r}})$, the gluing diffeomorphism in (\ref{Infinite Set R2}) is chosen so that it identifies the induced $\Pin^+$-structures on the common $S^2\widetilde{\times} S^1$ boundaries \cite[p. 160]{[Stolz]} (cf. \cite{[AitchisonRubinstein], [CappellShaneson]}). In particular, we can glue these structures together to obtain a $\Pin^+$-structure $(X_r, \phi_{X_r})$ and a $\Pin^+$-bordism\begin{equation}\label{Pin Bordism Theorem C}(V; X_r, X\sqcup M_{B_1\cdots B_k});\end{equation}see \cite{[HambletonKreckTeichner], [KirbyTaylor], [Stolz]}  for details. It follows that\begin{equation}\eta(X_r, \phi_{X_r}) = \eta(X, \phi_X) + \eta(M_{B_1\cdots B_r}, \phi_{M_{B_1\cdots B_r}}) \mod 2\Z\end{equation}given that the $\eta$-invariant is additive with respect to disjoint unions \cite[Section 7]{[Stolz]}. The next two values are inferred from Proposition \ref{Proposition Pin Structures Mapping Torus Generalized}.

$\bullet$ If $\eta(X, \phi_{X}) = \pm \frac{1}{2}\mod 2\Z$ or $r = 0 \mod 2$, then\begin{equation}\label{Identity 1}\eta(X_r, \phi_{X_r}) = \eta(X, \phi_X) \mod 2\Z.\end{equation}

$\bullet$ If $\eta(X, \phi_{X}) \neq \pm \frac{1}{2}\mod 2\Z$ and $r \neq 0 \mod 2$, then\begin{equation}\label{Identity 2}\eta(X_r, \phi_{X_r}) = \eta(X, \phi_X) + 1 \mod 2\Z\end{equation}and\begin{equation}\label{Identity 3}\eta(X_r, \phi_{X_r}) \neq \eta(X, \phi_X) \mod 2\Z.\end{equation}
If $\eta(X, \phi_X)\neq \pm \frac{1}{2}\mod 2\Z$ and $r \neq 0 \mod 2\Z$, it follows that there is no diffeomorphism\begin{equation}\label{Diffeo Theorem C}f:X_r\cs (n - 1)\cdot(S^2\times S^2)\rightarrow X\cs (n - 1)\cdot(S^2\times S^2)\end{equation} such that $f^\ast(\phi_{X\cs (n -1)\cdot(S^2\times S^2)}) = \phi_{X_r\cs (n - 1)\cdot(S^2\times S^2)}$; these are the $\Pin^+$-structures on the connected sums $X\cs (n - 1)\cdot(S^2\times S^2)$ and $X_r\cs (n - 1)\cdot(S^2\times S^2)$. The hypothesis $H^1(X; \Z/2) = \Z/2$ implies that there are only two $\Pin^+$-structures to consider, and they are mutual inverses in $\Omega^{\Pin^+}_4$ \cite[p. 190]{[KirbyTaylor]} as it was mentioned before. Therefore, there is no diffeomorphism (\ref{Diffeo Theorem C}) for any $n\in \N$ at all.\end{proof}

\begin{remark}\label{Remark Hypothesis H1}The hypothesis $H^1(X; \Z/2) = \Z/2$ of Theorem \ref{Theorem Mechanism} is technical in nature and its inclusion allows us to keep the exposition short. The hypothesis can be removed by observing that some of Stolz's results \cite{[Stolz]} can be extended to obstruct the existence of a diffeomorphism between two given 4-manifolds that admit more than two $\Pin^+$-structures whenever the values of their $\eta$-invariants differ for every choice of $\Pin^+$-structure.\end{remark}

\section{Questions and further work.}\label{Section Questions}

We finish the paper stating three canonical questions regarding the nonorientable 4-manifolds that were constructed.

\begin{qst}\label{Question Infinite}Is there a 4-manifold $X$ for which the set $\{X_r: r\in \N\}$ consists of pairwise non-diffeomorphic elements?
\end{qst}

A result of Gompf \cite{[Gompf1]} states that any two closed nonorientable homeomorphic 4-manifolds become diffeomorphic after connect summing with sufficiently large number copies of the complex projective plane. 

\begin{qst}Is $X_r\cs \CP^2$ diffeomorphic to $X\cs \CP^2$ for every $r\in \N$?
\end{qst}

Cappell-Shaneson's construction in \cite{[CappellShaneson]} need not yield inequivalent smooth structure. Akbulut \cite[Section 2]{[Akbulut0]} showed that an application of the method in \cite[Theorem 3.1]{[CappellShaneson]} to $\RP^2\times S^2$ does not produce an inequivalent smooth structure. Instead, their construction yields a homeomorphism $\RP^2\times S^2\rightarrow \RP^2\times S^2$ whose normal invariant is nontrivial \cite{[Akbulut0]}. Akbulut \cite[Theorem 2]{[Akbulut0]} also showed that there is a diffeomorphism $Q\cs \CP^2\rightarrow \RP^4\cs \CP^2$. These results suggest the following variation to Question \ref{Question Infinite}. 

\begin{qst}Is there a 4-manifold $X_{r_0}$ in the set $\{X_r: r\in \N\}$ that is not diffeomorphic to $X$ for $X = \RP^4\cs \CP^2$ or $\RP^2\times S^2$?
\end{qst}

\end{document}